\documentclass[12pt,twoside]{amsart}
 \usepackage{amsmath,amssymb,amsthm,amsfonts,latexsym}
 \usepackage{exscale,relsize}
 \usepackage[all]{xy}
 \usepackage[colorlinks,urlcolor=blue,linkcolor=blue,linktocpage=magenta]{hyperref}

 \newtheorem{theorem}{Theorem}[section]
 \newtheorem{lemma}[theorem]{Lemma}
 \newtheorem{proposition}[theorem]{Proposition}
 \newtheorem{corollary}[theorem]{Corollary}
 \theoremstyle{remark}
 
 \newtheorem{example}[theorem]{\bf Example}
 \newtheorem{remark}[theorem]{\bf Remark}
 
 \newtheorem{conjecture}[theorem]{\bf Conjecture}

 \numberwithin{equation}{section}

 \newcommand{\Z}{\mathbb{Z}}
 \newcommand{\Q}{\mathbb{Q}}
 \newcommand{\CC}{C}
 \newcommand{\F}{\mathbb{F}}
 \newcommand{\D}{\Delta}
 \renewcommand{\d}{\sigma}

 \newcommand{\ra}{\rightarrow}
 \newcommand{\bs}{\backslash}
 \newcommand{\wt}{\widetilde}
 \newcommand{\cd}{\circledast}
 
\begin{document}

 \baselineskip=0.7cm

 \title{Hamming Polynomial of a Demimatroid}
 \author{Jos\'e Mart\'{\i}nez-Bernal, Miguel A. Valencia-Bucio\ \ and\ \ Rafael H. Villarreal}
 \address{Departamento de Matem\'aticas, Cinvestav-IPN, M\'exico\\
 \scriptsize A.P. 14-740, Ciudad de M\'exico 07360}
 \email{\{jmb, mavalencia, vila\}@math.cinvestav.mx}
 \maketitle

{\bf Abstract} Following Britz, Johnsen, Mayhew and Shiromoto, we consider
demi\-ma\-troids as a(nother) natural generalization of matroids. As they have
shown, demi\-ma\-troids are the appropriate combinatorial objects for studying
Wei's duality. Our results here apport further evidence about the trueness of
that observation. We define the Hamming polynomial of a demimatroid $M$, denoted
by $W(x,y,t)$, as a generalization of the extended Hamming weight enumerator of
a matroid. The polynomial $W(x,y,t)$ is a specialization of the Tutte polynomial
of $M$, and actually is equivalent to it. Guided by work of Johnsen, Roksvold
and Verdure for matroids, we prove that Betti numbers of a demimatroid and its
elongations determine the Hamming polynomial. Our results may be applied to
simplicial complexes since in a canonical way they can be viewed as
demimatroids. Furthermore, following work of Brylawski and Gordon, we show how
demimatroids may be generalized one step further, to combinatroids. A
combinatroid, or Brylawski structure, is an integer valued function $\rho$,
defined over the power set of a finite ground set, satisfying the only condition
$\rho(\emptyset)=0$. Even in this extreme generality, we will show that many
concepts and invariants in coding theory can be carried on directly to
combinatroids, say, Tutte polynomial, characteristic polynomial, MacWilliams
identity, extended Hamming polynomial, and the $r$-th generalized Hamming
polynomial; this last one, at least conjecturelly, guided by the work of Jurrius
and Pellikaan for linear codes. All this largely extends the notions of
deletion, contraction, duality and codes to non-matroidal structures.

\newpage
\tableofcontents

\section{Introduction}

Matroids are combinatorial objects introduced by Whitney in 1935 as a
generalization of both graphs and matrices. They capture geometric and
combinatorial properties of linear independence over finite structures.
Demimatroids (Section~\ref{defdemi}) are a generalization of matroids, and in
what follows we will show how demimatroids may be generalized one step further
to combinatroids, via the rank function. We will show that combinatroids capture
many concepts related with duality in coding theory and matroids. For instance,
we define invariants as the Tutte polynomial, the generalized Hamming polynomial
and the extended Hamming polynomial; or relationships between them, as deletion,
contraction and the MacWilliams identity.

Denote by ${\mathcal C}$ the family of combinatroids defined over the same
ground set $E$, and by ${\mathcal D}$ the smaller subfamily of demimatroids. The
four operations: identity, dual, nullity and supplement
(Section~\ref{defsupplement}), may be seen as duality operators acting on
${\mathcal C}$, actually, these last three operators form a triality, in the
sense that the composition of two of them results in the third one. The
restriction to ${\mathcal D}$ of these operators behave even better: ${\mathcal
D}$ has a natural structure of a bounded distributive lattice, and each
demimatroid determines a weight hierarchy and a Duursma zeta function, which is
a largely extension of well-known results for linear codes. All these facts show
that ${\mathcal D}$ is a mathematical object that merits a further study.

As a final result, by extending work of Johnsen, Roksvold and Verdure for
matroids, we prove that Betti numbers of a demimatroid and its elongations
determine the extended Hamming polynomial of a demimatroid. All these results
may be applied to simplicial complexes since in a canonical way they can be
viewed as demimatroids. For unexplained notions of graph theory, linear codes
and matroids we refer to \cite{Harary-z}, \cite{Huffman-Pless-z} and
\cite{Oxley-z}, respectively.

\section{Matroids and linear codes}

A \emph{matroid} is a pair $M=(E,\rho)$, where $E$ is a finite set called the
\emph{ground set} of $M$, and $\rho:2^E\ra\Z_+:=\{0,1,\ldots\}$ is a function
satisfying:
 \begin{itemize}
 \item[$(R_0)$] $\rho(\emptyset)=0$;
 \item[$(R_1)$] If $X\subseteq E$ and $x\in E$, then
 $\rho(X)\leq\rho(X\cup\{x\})\leq\rho(X)+1$;
 \item[$(R_2)$] If $X,Y\subseteq E$, then $\rho(X\cup Y)+\rho(X\cap
 Y)\leq\rho(X)+\rho(Y)$.
 \end{itemize}

The function $\rho$ is called the \emph{rank function} of the matroid. Condition
$(R_2)$ is known as the \emph{submodularity} condition. An \emph{independent
set} of $M$ is a subset $X\subseteq E$ such that $\rho(X)=|X|$, where $|X|$
denotes the cardinality of $X$; in particular the empty set is always an
independent set. A \emph{basis} is an in\-clu\-sion maximal independent set; one
can verify that bases of a matroid are equicardinal. A subset of the ground set
which is not independent is called a \emph{dependent set}, and a \emph{circuit}
is a minimal dependent set.

Let $X$ be a subset of $E$. From $(R_0)$ and $(R_1)$, and by a direct induction
argument, if follows that $0\leq\rho(X)\leq|X|$ for all $X\subseteq E$. The
\emph{nullity} of $X$, denoted by $\eta(X)$, is defined as
$\eta(X):=|X|-\rho(X)$. In particular, the nullity of $M$ is defined as
$\eta(M):=\eta(E)$. The \emph{$r$-generalized Hamming weight} of the matroid $M$
is given by
\[d_r(M):=\min\{|X|:\eta(X)=r\},\quad\quad 1\leq r\leq \eta(E),\]
and the sequence $d_1(M),\ldots,d_{\eta(E)}(M)$ is called the \emph{weight
hierarchy} of $M$.

\bigskip Let $p$ be a prime, $q$ a positive power of $p$ and $\F_q$ a field with
$q$ elements. A \emph{linear} $[n,k]_q$ \emph{code} is a $k$-dimensional
subspace $C$ of $\F_q^n$. In this context the field $\F_q$ is called the
\emph{alphabet}, the elements of $\F_q^n$ are the \emph{words} and the elements
of $C$ are called \emph{codewords} of the code. We consider $\F_q^n$ provided
with its \emph{Hamming distance}, which is the number of coordinates in which
two words differ. For $c\in C$ its \emph{weight}, denoted by $w(c)$, is the
number of its nonzero coordinates. For a subset $X$ of $\F_q^n$ we define the
\emph{support} of $X$, denoted $\text{supp}(X)$, as the union of all the
supports of elements in $X$, i.e.
\[\text{supp}(X):=\{i:\exists (c_1,\ldots,c_n)\in C\ \text{such that } c_i\neq
0\}.\] Let $C$ be a linear $[n,k]_q$ code. For $1\leq r\leq k$, the $r$-th
\emph{generalized Hamming weight} of $C$ is defined as
\[d_r(C):=\min\{|\text{supp}(X)|:X\ \text{is a $r$-dimensional subspace of } C\}.\]
The number $d_1(C)$ is known as the \emph{minimum distance} of the code and the
sequence $d_1(\CC),\ldots,d_k(\CC)$ is called the \emph{weight hierarchy} of
$C$.

With each linear code $C$ one associate the vector matroid $M[H]$ on the ground
set $E=\{1,\ldots,n\}$, where $H$ is a parity check matrix of $C$. The rank
function of $M[H]$ is given by $\rho(X):=\text{rank}(H_X)$ for $X\subseteq E$,
where $H_X$ is the submatrix of $H$ obtained by picking the columns indexed by
$X$. The matroid $M[H]$ does not depend on the parity check matrix we use. We
call $M[H]$ the \emph{(parity) matroid} of $C$. A basic result in this area,
relating codes and matroids, is that the weight hierarchies of both the code $C$
and the matroid $M[H]$ coincide~\cite{Wei-z}.

\section{Demimatroids}

A \emph{demimatroid} is a pair $M=(E,\rho)$, where $E$ is a finite set called
the \emph{ground set} of $M$, and $\rho:2^E\ra\Z_+$ is a function such that
 \begin{itemize}\label{defdemi}
 \item[$(R_0)$] $\rho(\emptyset)=0$;
 \item[$(R_1)$] If $X\subseteq E$ and $x\in E$, then
 $\rho(X)\leq\rho(X\cup\{x\})\leq\rho(X)+1$;
 \end{itemize}
 The function $\rho$ is called the \emph{rank function} of the demimatroid.
 Clearly matroids are examples of demimatroids. By abuse of notation
 we will frequently refer to $\rho$ itself as the demimatroid. The
 \emph{rank} of a demimatroid $M$ is defined as $\rho(M):=\rho(E)$.
 A straightforward verification shows that $0\leq\rho(X)\leq|X|$ for all $X\subseteq E$.
 We define the \emph{nullity} of $X$ as $\eta(X):=|X|-\rho(X)$. The
 nullity of a demimatroid $M$ is defined as $\eta(M):=\eta(E)$.
The \emph{dual} of a demimatroid $M=(E,\rho)$ is the pair $M^*:=(E,\rho^*)$,
where
\[\rho^*(X):=|X|+\rho(E\bs X)-\rho(E).\] Clearly $\rho^*(\emptyset)=0$. To
simplify notation, from here one we will write $X\bs x$ and $X\cup x$ instead of
$X\bs\{x\}$ and $X\cup\{x\}$, respectively. If $x\in X$, obviously $(R_1)$ is
satisfied, and if $x\notin X$, then $\rho^*(X)\leq\rho^*(X\cup
x)\leq\rho^*(X)+1$ if and only if $\rho(E\bs X)\leq\rho((E\bs X)\bs
x)+1\leq\rho(E\bs X)+1$. But each of these last two inequalities readily follows
from the properties of $\rho$. So, in fact, $\rho^*$ is a demimatroid. Moreover,
one can verify that $M^{**}=M$; to see this just note that
$\rho(E)+\rho^*(E)=|E|$, and then
\[\rho^{**}(X)=|X|+\rho^*(E\bs X)-\rho^*(E)=|E|+\rho(X)-|E|
=\rho(X).\]

As in the case of matroids, we define \emph{independent sets} of a demimatroid
as those $X\subseteq E$ such that $\rho(X)=|X|$; and in a similar fashion, one
might define bases, dependent sets and circuits. But in this generality we must
remark that bases of a demimatroid are not necessarily equicardinal.

\begin{example} Let $E$ be a finite set and $\rho:2^E\ra\Z_+$
given by:
\begin{itemize}
 \item[(i)] $\rho(X)=0$ for all $X\subseteq E$. Then $\rho$ is a demimatroid;
called the \emph{trivial} demimatroid.

 \item[(ii)] $\rho(X)=|X|$ for all $X\subseteq E$. Then $\rho$ is a demimatroid; actually it
is a matroid.

 \item[(iii)] $\rho(X)=0$ if $X\neq E$ and $\rho(E)=1$. Then $\rho$ is a demimatroid;
if $E$ has at least two elements, then $\rho$ is not a matroid.

 \item[(iv)] $\rho(\emptyset)=0$ and $\rho(X)=1$ for all $X\neq \emptyset$. Then $M=(E,\rho)$
is a demimatroid.
\end{itemize}
\end{example}

\begin{example} Let $M=(E,\rho)$ be a nontrivial dematroid.
For $X\subseteq E$ define $\rho^\bullet(X)=\rho(X)$ if $\rho(X)<\rho(E)$ and
$\rho^\bullet(X)=\rho(X)-1$ if $\rho(X)=\rho(E)$. Then $(E,\rho^\bullet)$ is a
demimatroid.\end{example}

A \emph{simplicial complex} $\D$ on a finite vertex set $E$ is an inclusion
closed family of subsets of $E$, i.e. $\sigma\in\D$ and $\tau\subseteq\sigma$
implies $\tau\in\D$. Elements of $\D$ are called \emph{faces} and maximal faces
are called \emph{facets}. A face of $\D$ whose cardinality is $i+1$ is said to
be of \emph{dimension} $i$. The \emph{dimension} of $\D$ is the maximum
dimension of any one of its faces.

\begin{example}\label{updemi} Let $\D$ be a simplicial complex on the vertex set $E$.
We define the demimatroid $\D^\uparrow:=(E,\rho)$, where, for all $X\subseteq
E$, \[\rho(X):=\max\{|\sigma|: \sigma\subseteq X, \sigma\in\D\}.\]\end{example}

\begin{example} A graph may be viewed as a $1$-dimensional simplicial complex,
and then as a demimatroid. Say the graph has no isolated vertices and let $E$
denote the vertex set. Thus, in this case, the demimatroid in
Example~\ref{updemi} is given by $\rho(\emptyset)=0$, $\rho(X)=1$ if $X$ is and
independent vertex set of $G$, and $\rho(X)=2$ if $X$ is not an independent
vertex set of $G$.
\end{example}

\begin{example} Let $M=(E,\rho)$ be a demimatroid. If $\rho(X)=|X|$ for some
$X\subseteq E$, then $\rho(Y)=|Y|$ for all $Y\subseteq X$. Therefore, the set
\[M^\downarrow:=\{X\subseteq E:\rho(X)=|X|\}\] is a simplicial
complex.\end{example}

\begin{example} Let $\D$ be a simplicial complex with vertex set $E$ and
$\rho:2^E\ra\Z_+$ given by $\rho(X)=|X|$ if $X\in\D$ and $\rho(X)= |X|-1$ if
$X\notin\D$. Then $\D^\sharp:=(E,\rho)$ is a demimatroid.\end{example}

\begin{example}\label{nullsc} Let $M=(E,\rho)$ be a demimatroid. Since $\rho$ is
non-decreasing, it follows that, for all nonnegative integers $r$, the set
$M_{(r)}:=\{X\subseteq E:\rho(X)\leq r\}$ is a simplicial complex.\end{example}

Let $E$ be a finite set. Denote by ${\mathcal S}$ the family of all simplicial
complexes with ground set $E$, and make ${\mathcal S}$ a poset defining
$\D\leq\Gamma$ when $\Delta\subseteq\Gamma$. Denote by ${\mathcal D}$ the family
of all demimatroids with ground set $E$, and make ${\mathcal D}$ a poset by
defining $(E,\rho)\leq (E,\tau)$ when $\rho(X)\leq\tau(X)$ for all $X\subseteq
E$. The next lemma is not hard to prove.

\begin{lemma}
\begin{itemize}
 \item[(i)] $\D\leq\Gamma$ implies $\D^\uparrow\leq\Gamma^\uparrow$;
 \item[(ii)] $(E,\rho)\leq (E,\tau)$ implies $(E,\rho)^\downarrow\leq
 (E,\tau)^\downarrow$;
 \item[(iii)] $(\D^\uparrow)^\downarrow=\D$.
 \item[(iv)] $M^\downarrow=\D$ implies $\D^\uparrow\leq M$; in particular,
 $(M^\downarrow)^\uparrow\leq M$.
 \item[(v)] $(M^\downarrow)^\uparrow=M$ if and only if $M=\D^\uparrow$ for some
 simplicial complex $\D$.
\end{itemize}
\end{lemma}

\begin{proof} (iv): Say $M=(E,\rho)$ and $\D^\uparrow=(E,\tau)$. Take any
$X\subseteq E$. $\tau(X)=\max\{|\sigma|:\sigma\subseteq X,\
\rho(\sigma)=|\sigma|\}$. Choose $\sigma\subseteq X$ such that
$\tau(X)=|\sigma|$ and $\rho(\sigma)=|\sigma|$. Then
$\tau(X)=|\sigma|=\rho(\sigma)\leq\rho(X)$.
\end{proof}

\begin{example} Let $\D$ be a simplicial complex and $M$ a demimatroid. Then
$M^\downarrow=\D$ if and only if $\D^\uparrow\leq M\leq\D^\sharp$.\end{example}

A \emph{Galois connection} between two posets $P$ and $Q$ is a pair of functions
$\alpha:P\ra Q$ and $\beta:Q\ra P$ with the properties: (1) both $\alpha$ and
$\beta$ are order-inverting; (2) $p\leq\beta(\alpha(p))$ for all $p\in P$ and
$q\leq\alpha(\beta(q))$ for all $q\in Q$.

\begin{proposition} Let ${\mathcal D}^\text{\rm op}$ denote the dual poset of ${\mathcal D}$.
The maps $^\uparrow:{\mathcal S}\ra{\mathcal D}'$, $\D\mapsto\D^\uparrow$ and
$^\downarrow:{\mathcal D}^\text{\rm op}\ra{\mathcal S}$, $M\mapsto M^\downarrow$
form a Galois connection.\end{proposition}

\section{Combinatroids}

Three important operations on matroids are motivated by graph theory: deletion,
contraction and duality. Brylawski realized that it is possible to extend all of
these three operations to any finite set $E$ provided with an arbitrary function
$r:2^E\ra\Z$, see~\cite{Gordon-z}. Thus we define a \emph{combinatroid} (with
values in $\Z$) as a pair $M:=(E,\rho)$, where $E$ is a finite set called the
\emph{ground set} of $M$, and $\rho:2^E\ra\Z$ is a function satisfying the only
condition $\rho(\emptyset)=0$. The function $\rho$ is called the \emph{rank
function} of the combinatroid. Clearly demimatroids are examples of
combinatroids. Another name for a combinatroid is a (normalized) \emph{Brylawski
structure}, as is done in~\cite{Gordon-z}. One define the \emph{dual}
com\-bi\-na\-troid $M^*=(E,\rho^*)$, where $\rho^*$, called the \emph{dual rank
function}, is given by
\[\rho^*(X)=|X|+\rho(E\bs X)-\rho(E).\]

Then, the \emph{deletion} of $A\subseteq E$, denoted by $M\bs A$, is defined as
the restriction of the rank function $\rho$ to $E\bs A$, i.e. $\rho_{M\bs
A}(X):=\rho(X)$ for all $X\subseteq E\bs A$. Moreover, \emph{contraction},
denoted by $M/A$, is defined using deletion and duality: $M/A:=(M^*\bs A)^*.$
Note that both $M\bs A$ and $M/A$ have the same ground set $E\bs A$.
\begin{proposition} {\rm(Brylawski, Gordon; see~\cite{Gordon-z})} Let $M=(E,\rho)$ be
a combinatroid and $A\subseteq E$.
\begin{itemize}
 \item[(i)] $\rho_{M/A}(X)=\rho(X\cup A)-\rho(A)$ for all $X\subseteq E\bs A$;
 \item[(ii)] $(M^*)^*=M$;
 \item[(iii)] $(M\bs A)^*=M^*/A$;
 \item[(iv)] $(M/A)^*=M^*\bs A$.
\end{itemize}
\end{proposition}

\begin{proposition} Let $M=(E,\rho)$ be a demimatroid and $A\subseteq E$. Then
\begin{itemize}
 \item[(i)] $M\bs A$ is a demimatroid;
 \item[(ii)] $M/A$ is a demimatroid.
\end{itemize}\end{proposition}

\begin{proof} Let $X\subset E\bs A$ and $x\in (E\bs A)\bs X$.

(i): $\rho_{M\bs A}(X)\leq\rho_{M\bs A}(X\cup x)\leq\rho_{M\bs A}(X)+1$
$\Leftrightarrow$ $\rho(X)\leq\rho(X\cup x)\leq\rho(X)+1$.

(ii): $\rho_{M/A}(X)\leq\rho_{M/A}(X\cup x)\leq\rho_{M/A}(X)+1$
$\Leftrightarrow$ $\rho(X\cup A)-\rho(A)\leq\rho(X\cup x\cup
A)-\rho(A)\leq\rho(X\cup A)-\rho(A)+1$ $\Leftrightarrow$ $\rho(X\cup
A)\leq\rho(X\cup A\cup x)\leq\rho(X\cup A)+1$.\end{proof}
\medskip A \emph{minor} of a demimatroid $M$ is any demimatroid obtainded from $M$
by a sequence of deletions and contractions.

\bigskip One can also define the \emph{nullity} combinatroid
$M^\circ=(E,\rho^\circ)$, where $\rho^\circ$, called the \emph{nullity
function}, is given by
\[\rho^\circ(X)=|X|-\rho(X).\]
\begin{proposition} Let $M=(E,\rho)$ be a combinatroid. Then
\begin{itemize}
 \item[(i)] $\rho^{*\circ}(X)=\rho^{\circ *}(X)=
 \rho(E)-\rho(E\bs X)$\quad for all $X\subseteq E$;
 \item[(ii)] $(M^\circ)^\circ=M$;
 \item[(iii)] $(M^*)^\circ=(M^\circ)^*$;
 \item[(iv)] If $M$ is a demimatroid, then $M^\circ$ is a demimatroid.
\end{itemize}
\end{proposition}

\proof (iv): Obviously $\rho^\circ(\emptyset)=0$. Let $X\subset E$ and $x\in
E\bs X$. $\rho^\circ(X)\leq\rho^\circ(X\cup x)\leq\rho^\circ(X)+1$ if and only
if $|X|-\rho(X)\leq|X|+1-\rho(X\cup x)\leq|X|-\rho(X)+1$ if and only if
$\rho(X)+1\geq\rho(X\cup x)\geq\rho(X)$.\qed

\medskip Following~\cite{BJMS-z}, we define the \emph{supplement}
combinatroid $M^\cd:=(E,\rho^\cd)$, where $\rho^\cd$, called the
\emph{supplement (or supplementary) function}, is given by
\[\rho^\cd(X)=\rho(E)-\rho(E\bs X).\]\label{defsupplement}
\begin{proposition} Let $M=(E,\rho)$ be a combinatroid. Then
\begin{itemize}
 \item[(i)] $(M^\cd)^\cd=M$;
 \item[(ii)] $(M^*)^\circ=(M^\circ)^*=M^\cd$;
 \item[(iii)] $(M^\circ)^\cd=(M^\cd)^\circ=M^*$;
 \item[(iv)] $(M^*)^\cd=(M^\cd)^*=M^\circ$;
 \item[(v)] {\rm(\cite[Thm. 8]{BJMS-z})} If $M$ is a demimatroid, then
 $M^\cd$ is a demimatroid.
\end{itemize}
\end{proposition}
\begin{proof} (v): Obviously $\rho^\cd(\emptyset)=0$. Let $X\subset E$ and $x\in
E\bs X$. $\rho^\cd(X)\leq\rho^\cd(X\cup x)\leq\rho^\cd(X)+1$ if and only if
$\rho(E)-\rho(E\bs X)\leq\rho(E)-\rho(E\bs(X\cup x))\leq\rho(E)-\rho(E\bs X)+1$
if and only if $\rho(E\bs X)\geq\rho((E\bs X)\bs x)\geq\rho(E\bs X)-1$. But each
of these last two inequalities directly follows from the properties of
$\rho$.\end{proof}

\medskip The identity (denoted by ``id"), dual, nullity and supplement
operations may be viewed as operators acting on the set of combinatroidal
structures defined on the same ground set $E$.

\begin{proposition} Let $M=(E,\rho)$ be a combinatroid. Then the operators
$\{\text{\rm id},*,\circ,\cd\}$ form an abelian group isomorphic to
$\Z_2\times\Z_2$:
\[\begin{array}{|c||c|c|c|c|}\hline
  & \text{\rm id}  & * & \circ& \cd\\\hline\hline
 \text{\rm id}& \text{\rm id} & * & \circ & \cd\\\hline
 *& * & \text{\rm id} & \cd & \circ\\\hline
 \circ& \circ & \cd & \text{\rm id} & *\\\hline
 \cd& \cd & \circ & * & \text{\rm id} \\\hline
\end{array}\,.\]\end{proposition}

\begin{remark} Note that the operators $\{*,\circ,\cd\}$ form a \emph{triality},
in the sense that the composition of two of them gives the third
one.\end{remark}

\begin{example} Let $E$ be a finite set and $\rho:2^E\ra\Z_+$, $X\mapsto|X|$.
Then $\rho^*\equiv\rho^\circ\equiv0$ and $\rho^\cd\equiv\rho$.\end{example}

\begin{example} Let $E$ be a finite set and $\rho:2^E\ra\Z_+$, $\rho(X)=0$ if
$X\neq E$ and $\rho(E)=1$. We have that $\rho^*(\emptyset)=0$ and
$\rho^*(X)=|X|-1$ if $X\neq\emptyset$; $\rho^\circ(X)=|X|$ if $X\neq E$ and
$\rho^\circ(E)=|E|-1$; $\rho^\cd(\emptyset)=0$ and $\rho^\cd(X)=1$ if
$X\neq\emptyset$.\end{example}

\begin{example}\label{Contraex} Let $M=(E=\{1,2,3\},\rho)$ be the matroid
whose basis are $\{1,2\}$ and $\{1,3\}$. We have the following table:
\[\begin{array}{|r|c|c|c|c|c|c|c|c|c|}\hline
 X & \emptyset&1&2&3&12&13&23&E\\\hline\hline
 \rho&0&1&1&1&2&2&1&2\\\hline
 \rho^*&0&0&1&1&1&1&1&1\\\hline
 \rho^\circ&0&0&0&0&0&0&1&1\\\hline
 \rho^\cd&0&1&0&0&1&1&1&2\\\hline
\end{array}\,.\]
\end{example}

\begin{remark} If $M$ is a matroid, then $M^\circ$ and $M^\cd$ are demimatroids,
but they might not be matroids. For instance, in Example~\ref{Contraex},
$1=\rho^\circ(23)+\rho^\circ(\emptyset)\not\leq\rho^\circ(2)+\rho^\circ(3)=0$
and $1=\rho^\cd(23)+\rho^\cd(\emptyset)\not\leq\rho^\cd(2)+\rho^\cd(3)=0$, show
that $\rho^\circ$ and $\rho^\cd$ do not satisfy the submodularity
condition.\end{remark}

\begin{example} Let $G$ be a simple graph with no isolated vertices; we see
$G$ as a $1$-dimensional simplicial complex. Let $E$ denote the vertex set of
$G$ and define $\rho:2^E\ra\Z_+$, $\rho(\emptyset)=0$, $\rho(X)=1$ if $X$ is and
independent vertex set of $G$, and $\rho(X)=2$ if $X$ is not an independent
vertex set of $G$. Then
\begin{eqnarray*}
\rho^*(X)&=&\begin{cases}
|X|,& \text{if $X$ is not a covering};\\
|X|-1,& \text{if $X$ is a covering};\\
|X|-2,& \text{if $X=E$}.\end{cases}\\
\rho^\circ(X)&=&\begin{cases}
0,& \text{if $X=\emptyset$};\\
|X|-1,& \text{if $X$ is idependent};\\
|X|-2,& \text{if $X$ is not independent}.\end{cases}\\
\rho^\cd(X)&=&\begin{cases}
0,& \text{if $X$ is not a covering};\\
1,& \text{if $X$ is a covering}.\end{cases}
\end{eqnarray*}
\end{example}

For $\alpha$ and $\beta$, combinatroids over $E$, we define
$(\alpha\vee\beta)(X)=\max\{\alpha(X),\beta(X)\}$ and
$(\alpha\wedge\beta)(X)=\min\{\alpha(X),\beta(X)\}$ for all $X\subseteq E$.

\begin{lemma} If $\alpha$ and $\beta$ are demimatroids, then $\alpha\vee\beta$ and
$\alpha\wedge\beta$ are demimatroids.\end{lemma}

\begin{proof} This follows immediately from the fact that for real numbers
$a_1\leq a_2$ and $b_1\leq b_2$ it holds that
$\min\{a_1,b_1\}\leq\min\{a_2,b_2\}$ and
$\max\{a_1,b_1\}\leq\max\{a_2,b_2\}$.\end{proof}

The set of combinatroids on a set $E$ may be partially ordered by defining
$\alpha\leq\beta$ if $\alpha(X)\leq\beta(X)$ for all $X\subseteq E$.

\begin{theorem}\label{bounded-lattice}
The set of demimatroides on a finite set $E$, with $\vee$ and $\wedge$ defined
as above, form a bounded distributive lattice. The maximum demimatroid is
$|\cdot|: X\mapsto |X|$ and the minimum demimatroid is $0:X\mapsto
0$.\end{theorem}

\begin{example} This lattice has only one atom, namely, $\rho:2^E\ra\Z_+$,
$\rho(X)=0$ if $X\neq E$ and $\rho(E)=1$. And it also has only one coatom, which
is the nullity of $\rho$, i.e. $\rho^\circ(X)=|X|$ for all $X\neq E$ and
$\rho^\circ(E)=|E|-1$.
\end{example}

Let $M=(E,\rho)$ be a nontrivial demimatroid, and set $k:=\rho(E)\leq|E|$.
Define $\d_k(M):=\min\{|X|: \rho(X)=k\}$ and choose $X\subseteq E$ such that
$\d_k(M)=|X|$. For $x\in X$ we know that $\rho(X\bs x)<\rho(X)\leq\rho(X\bs
x)+1$. From this it follows that $\rho(X\bs x)=k-1$. Define
$\d_{k-1}(M):=\min\{|Y|:\rho(Y)=k-1\}$ and choose $Y\subseteq E$ such that
$\d_{k-1}(M)=|Y|$. For $y\in Y$ we know that $\rho(Y\bs y)<\rho(Y)\leq\rho(Y\bs
y)+1$. From this it follows that $\rho(Y\bs y)=k-2$. Continuing this process we
obtain that $0=\d_0(M)<\d_1(M)<\cdots<\d_k(M)\leq|E|$. A subset $X$ of $E$ is
said to be of \emph{level} $r$ if $\rho(X)=r$. Thus $\rho$ induce a partition of
$2^E$ by level sets. We put this on record as the following lemma, but first a
definition. For $1\leq r\leq\rho(E)$ we define the $r$-th \emph{Wei number} of
the demimatroid as
\begin{equation}\label{Wei-number} \d_r(M):=\min\{|X|:\rho(X)=r\}.\end{equation}
\begin{lemma}\label{strat} Let $M=(E,\rho)$ be demimatroid of rank $k:=\rho(M)$. Then
\begin{itemize}\label{gSingleton}
 \item[(i)] The image of $\rho$ is the set $\{0,1,\ldots,k\}$;
 \item[(ii)] $0<\d_1(M)<\cdots<\d_k(M)\leq|E|$;
 \item[(iii)] If $\rho(X)\geq r$, then $|X|\geq \d_r(M)$;
 \item[(iv)] $\min\{|X|:\rho(X)=r\}\ =\min\{|X|:\rho(X)\geq r\}$.
 \item[(v)] {\rm(Generalized Singleton bound)} For all $0\leq r\leq k$ it holds
 that
 \[k+\d_r(M)\leq|E|+r.\]
\end{itemize}\end{lemma}
\proof (iii): Say $\rho(X)=r+s$. Then $|X|\geq \d_{r+s}(M)\geq \d_r(M)$.

 (v): $k+\d_k(M)\leq|E|+k$ iff $\d_k(M)\leq|E|$, which is true. Suppose the result
 is true for $r,\ldots,k$. Hence $k+\d_{r-1}(M)\leq k+\d_r(M)-1\leq|E|+r-1$.\qed

\begin{example}\label{Contra} Let $M=(E=\{1,2,3,4\},\rho)$ be the matroid whose basis are
$\{1,2\}$, $\{1,3\}$, $\{1,4\}$, $\{2,3\}$, $\{3,4\}$. We have the following
table:
\[\begin{array}{|r|c|c|c|c|c|c|c|c|c|c|c|c|c|c|c|c|}\hline
 X&\emptyset&1&2&3&4&12&13&14&23&24&34&123&124&134&234&E\\\hline\hline
 \rho&0&1&1&1&1&2&2&2&2&1&2&2&2&2&2&2\\\hline
 \rho^*&0&1&1&1&1&2&1&2&2&2&2&2&2&2&2&2\\\hline
 \rho^\circ&0&0&0&0&0&0&0&0&0&1&0&1&1&1&1&2\\\hline
 \rho^\cd&0&0&0&0&0&0&1&0&0&0&0&1&1&1&1&2\\\hline
 \end{array},
 \begin{array}{|c|c|c|}\hline
 &\d_1&\d_2\\\hline\hline
 \rho&2&4\\\hline
 \rho^*&2&4\\\hline
 \rho^\circ&1&2\\\hline
 \rho^\cd&1&2\\\hline
 \end{array}\,.\]
\end{example}

\bigskip Since $\rho^\cd(E\bs X)=\rho^\cd(E)-\rho(X)$ for all $X\subseteq E$, Eq.
(\ref{Wei-number}) can be rewritten as
\[\d_r(M)=\min\{|X|:\rho^\cd(E\bs X)=\rho^\cd(E)-r\}.\]
Thus we may interpret the $r$-th Wei number $\d_r(M)$ as the minimum number of
elements that must be removed from $E$ to decrease the rank of $M^\cd$ by $r$. A
fundamental result is the following.
\begin{theorem}{\rm(Wei's duality~\cite[Thm. 13]{BJMS-z})}\label{wei-duality}
Let $M=(E,\rho)$ be a demimatroid. Then, with $n=|E|$ and $k=\rho(M)$,
\[\{\d_1(M),\ldots,\d_k(M)\}=\{1,\ldots,n\}\setminus\{n+1-\d_1(M^*),\ldots,n+1-
\d_{n-k}(M^*)\}.\]
\end{theorem}

\proof Suppose $\d_i(M)=n+1-\d_j(M^*)$ for some $i,j$. Choose $X\subseteq E$
such that $\rho(X)=i$ and $\d_i(M)=|X|$. Hence $|E\bs X|=n-|X|=\d_j(M^*)-1$. By
Lemma~\ref{strat}(iii) we have that $\rho^*(E\bs X)\leq j-1$. Similarly, choose
$Y\subseteq E$ such that $\rho^*(Y)=j$ and $\d_j(M^*)=|Y|$. Hence $\rho(E\bs
Y)\leq i-1$. But this implies that $i+j-1=\rho^*(E\bs X)+\rho(E\bs Y)\leq
i+j-2$, which is not possible.\qed

\begin{remark} In the literature, $\d_r(M^\circ)$ is known as the $r$-th generalized Hamming
weight of $M$, and since $(M^\circ)^\cd=M^*$, then $\d_r(M^\circ)$ is the
minimum number of elements that must be removed from $E$ to decrease the rank of
$M^*$ by $r$.\end{remark}

\begin{remark} $\min\{|X|:\eta(X)=r\}+\max\{|Y|:\rho^*(E)-\rho^*(Y)=r\}=|E|$.
\end{remark}
\noindent{\it Proof of the Remark}. Set $a=\min\{|X|:\eta(X)=r\}$ and
$b=\max\{|Y|:\rho^*(Y)=\rho^*(E)-r\}$. Choose $X$ such that $a=|X|$. Since
$\rho^*(E\bs X)=\rho(E)-r$, it holds that $|E\bs X|\leq b$, so $|E|\leq a+b$. To
prove the other direction choose $Y$ such that $b=|Y|$. Since $\eta(E\bs Y)=r$,
it holds that $a\leq|E\bs Y|$, so $a+b\leq|E|$.\qed

\medskip Let $M=(E,\rho)$ be a demimatroid. From the Singleton bound we obtain
that $\d_1(M)\leq|E|-\rho(E)+1$. When equality is attained, $M$ is called a
\emph{full} demimatroid.

\begin{corollary} Let $M=(E,\rho)$ be a demimatroid, with $n=|E|$ and $k=\rho(E)$.
\begin{itemize}
 \item[(i)] If $k+\d_r(M)=n+r$, then $k+\d_s(M)=n+s$ for all $s\geq r$.
 \item[(ii)] If $M$ is full, then $M^*$ is full.
\end{itemize}\end{corollary}

\proof (i): The result is true for $s=r$. If it is true for $r,\ldots,s$, then
$k+\d_{s+1}(M)\geq k+\d_s(M)+1=n+s+1$.

(ii): By (i), $n+1-\d_r(M)=k-r+1$. Thus, by Wei's duality, $\d_s(M^*)=k+s$ for
$1\leq s\leq n-k$. In particular, $\d_1(M^*)=k+1=n-(n-k)+1=n-\rho^*(E)+1$.\qed

\begin{example}\label{Contra-full} Let $M=(E=\{1,2,3\},\rho)$ be the demimatroid, with $\rho$
given by:
\[\begin{array}{|r|c|c|c|c|c|c|c|c|}\hline
 X&\ \emptyset&\ 1&\ 2&\ 3&12&13&23&E\\\hline\hline
 \rho&0&0&0&0&1&1&1&2\\\hline
 \rho^*&0&0&0&0&0&0&0&1\\\hline
 \rho^\circ&0&1&1&1&1&1&1&1\\\hline
 \rho^\cd&0&1&1&1&2&2&2&2\\\hline
 \end{array}\,,\qquad
 \begin{array}{|c|c|c|}\hline
 &\d_1&\d_2\\\hline\hline
 \rho&2&3\\\hline
 \rho^*&3&\\\hline
 \rho^\circ&1&\\\hline
 \rho^\cd&1&2\\\hline
 \end{array}\,.\]
We observe that $\rho$ and $\rho^*$ are full, whereas $\rho^\circ$ and
$\rho^\cd$ are not.
\end{example}

\begin{lemma} Let $M=(E,\rho)$ be a demimatroid, with $n=|E|$ and $k=\rho(E)$.
Then $M$ is full if and only if \[\rho(X)=\begin{cases}
 0,&\text{if }\ |X|\leq n-k;\\
 r,&\text{if }\ |X|=n-k+r\text{ and } r\geq 1.
 \end{cases}\]\end{lemma}

\begin{proof} $(\Leftarrow)$ Evidently $\d_1(M)=\min\{|X|:\rho(X)=1\}=n-k+1$.

$(\Rightarrow)$ By Lemma~\ref{gSingleton}(v), $\d_s(M)=n-k+s$ for all $1\leq
s\leq k$. In particular, $\d_k(M)=n$ implies $\rho(E\bs x)\leq k-1$ for $x\in
E$. Let $X\subseteq E$ with $|X|=n-1$. Suppose that $\rho(X)\leq k-2$. Then
$\rho(X)\leq \rho(E)\leq\rho(X)+1\leq k-1$, which is not possible. Thus
$\rho(X)=k-1$. Suppose that if $|X|=n-k+r$, then $\rho(X)=s$. Let $X$ such that
$|X|=n-k+r-1$. If $\rho(X)\leq n-k+r-2$, then $\rho(X)\leq\rho(X\cup
x)\leq\rho(X)+1$, i.e. $r\leq r-1$, which is not possible.
\end{proof}

\begin{lemma}\label{demi-uniform}
Let $M=(E,\rho)$ be a full demimatroid, with $n=|E|$ and $k=\rho(E)$. Then
\begin{itemize}
 \item[(i)] $\rho^*(X)=\begin{cases}
 0,&\text{if }\ |X|\leq k;\\
 r,&\text{if }\ |X|=k+r\text{ and } r\geq 1.
 \end{cases}$
 \item[(ii)] $\rho^\circ(X)=\begin{cases}
 |X|,&\text{if }\ |X|\leq n-k;\\
 n-k,&\text{if }\ |X|>n-k.\end{cases}$
 \item[(iii)] $\rho^\cd(X)=\begin{cases}
 |X|,&\text{if }\ |X|\leq k;\\
 k,&\text{if }\ |X|>k.
 \end{cases}$
 \end{itemize}\end{lemma}

We said that a demimatroid $M$ is \emph{uniform} when $M^\circ$ is full.

\begin{corollary} Let $M=(E,\rho)$ be a full demimatroid, with
$n=|E|$ and $k=\rho(E)$. Then $M^\circ$ and $M^\cd$ are the uniform matroids of
rank $n-k$ and $k$, respectively.
\end{corollary}

\medskip The Wei numbers $\{\d_1(M),\ldots,\d_{\rho(M)}(M)\}$ of a demimatroid
$M=(E,\rho)$ determine a subset of $\{1,\ldots,|E|\}$. The reciprocal is also
true.
\begin{proposition} Let $E$ be a finite set and
$\{\d_1<\ldots<\d_k\}\subseteq\{1,\ldots,|E|\}$. Then there exists
$\rho:2^E\ra\{0,1,\ldots,k\}$ such that $M=(E,\rho)$ is a demimatroid,
$k=\rho(E)$ and $\d_r(M)=\d_r$ for all $1\leq r\leq \rho(E)$.\end{proposition}

\proof Put $\d_0:=0$, $\d_{k+1}:=|E|$ and define $\rho(X)=i$ if
$\d_i\leq|X|<\d_{i+1}$. Then $\d_r(M)=\min\{|X|:\rho(X)=r\}=\d_r$.\qed

\medskip Let $M=(E,\rho)$ be a nontrivial demimatroid, and set
$k:=\rho(E)\leq|E|$. Define $\d^0(M):=\max\{|X|: \rho(X)=0\}$ and choose
$X\subseteq E$ such that $\d^0(M)=|X|$. For $x\not\in X$ we know that $\rho(X)
<\rho(X\cup x)\leq\rho(X)+1$. From this it follows that $\rho(X\cup x)=1$.
Define $\d^1(M):=\max\{|Y|:\rho(Y)=1\}$ and choose $Y\subseteq E$ such that
$\d^1(M)=|Y|$. For $y\not\in Y$ we know that $\rho(Y)<\rho(Y\cup
y)\leq\rho(Y)+1$. From this it follows that $\rho(Y\cup y)=2$. Continuing this
process we obtain that $0\leq \d^{0}(M)<\d^1(M)<\cdots<\d^k(M)=|E|$. We again
put this on record as the following lemma, but first a definition. For $1\leq
r\leq\rho(E)$ we define the $r$-th \emph{upper Wei number} of the demimatroid as
\begin{equation}\label{upper-Wei-number} \d^r(M):=\max\{|X|:\rho(X)=r\}.\end{equation}
\begin{lemma}\label{upper-strat} Let $M=(E,\rho)$ be demimatroid of rank $k:=\rho(M)$. Then
\begin{itemize}
 \item[(i)] The image of $\rho$ is the set $\{0,1,\ldots,k\}$;
 \item[(ii)] $0\leq \d^0(M)<\cdots<\d^k(M)=|E|$;
 \item[(iii)] If $\rho(X)\leq r$, then $|X|\leq \d^r(M)$;
 \item[(iv)] $\max\{|X|:\rho(X)=r\}\ =\max\{|X|:\rho(X)\leq r\}$.
 \item[(v)] {\rm(Generalized upper Singleton bound)} For all $0\leq r\leq k$ it holds
 that
 \[k+\d^r(M)\leq|E|+r.\]
\end{itemize}\end{lemma}
\proof (iii): Say $\rho(X)=r-s$. Then $|X|\leq \d^{r-s}(M)\leq \d^r(M)$.

 (v) $k+\d^k(M)\leq|E|+k$ iff $\d^k(M)\leq|E|$, which is true. Suppose the result
 is true for $r,\ldots,k$. Hence $k+\d^{r-1}(M)\leq k+\d^r(M)-1\leq|E|+r-1$.\qed

 \begin{theorem}{\rm(Upper Wei's duality~\cite[Thm. 12]{BJMS-z})}\label{upper-wei-duality}
Let $M=(E,\rho)$ be a demimatroid. Then, with $n=|E|$ and $k=\rho(M)$,
\[\{\d^0(M)+1,\ldots,\d^{k-1}(M)+1\}=\{1,\ldots,n\}\setminus\{n-\d^0(M^*),\ldots,n-
\d^{n-k-1}(M^*)\}.\]
\end{theorem}

\proof Suppose $\d^i(M)+1=n-\d^j(M^*)$ for some $i,j$. Choose $X\subseteq E$
such that $\rho(X)=i$ and $\d^i(M)=|X|$. Hence $|E\bs X|=n-|X|=\d^j(M^*)+1$. By
Lemma~\ref{upper-strat}(iii) we have that $\rho^*(E\bs X)\geq j+1$. Similarly,
choose $Y\subseteq E$ such that $\rho^*(Y)=j$ and $\d^j(M^*)=|Y|$. Hence
$\rho(E\bs Y)\geq i+1$. But this implies that $i+j+1=\rho^*(E\bs X)+\rho(E\bs
Y)\geq i+j+2$, which is not possible.\qed

\section{Tutte polynomial}

\bigskip The Tutte polynomial is an important invariant for graphs and matroids.
We define the \emph{Tutte polynomial} of a combinatroid $M=(E,\rho)$ as
\begin{equation}\label{TuttePolynomial}
T_M(x,y):=\sum_{A\subseteq E}(x-1)^{\rho(E)-\rho(A)}(y-1)^{|A|-\rho(A)}.
 \end{equation}
Using the classical notation $\eta:=\rho^\circ$, this can be rewritten as
\begin{equation}\label{TuttePolynomialNulity}
T_M(x,y)=\sum_{A\subseteq E}(x-1)^{\eta^*(E\bs A)}(y-1)^{\eta(A)}.
 \end{equation}

\begin{remark} Since a combinatroid $\rho$ may take negative values, we must
remark that $T_M(x,y)$, as defined above, could be a rational function; so, a
better name would be the Tutte enumerator or the Tutte rational function.
However, since we will not use its properties as a rational function, by abuse
of language, we will continuous refering to it as the Tutte polynomial. On the
other hand, if $\rho$ is a demimatroid, then $T_M(x,y)$ is in fact a
polynomial.\end{remark}

\bigskip This Tutte polynomial is well-behaved with respect to combinatroidal
duality:
\begin{proposition} {\rm (Tutte duality)} Let $M=(E,\rho)$ be a combinatroid. Then
\[T_{M^*}(x,y)=T_M(y,x).\]\end{proposition}

\proof It follows immediately from Eq. (\ref{TuttePolynomialNulity}) by noticing
that $(\rho^*)^\circ=(\rho^\circ)^*=\eta^*$.\qed

\medskip Let $M=(E,\rho)$ be a combinatroid with Tutte polynomial $T_M(x,y)$. We
define its \emph{Hamming polynomial} by:
\begin{equation}\label{WHP-definition}
W_M(x,y,t):=(x-y)^{\eta(M)}y^{\rho(M)}\,T_M(\frac{x}{y},\frac{x+(t-1)y}{x-y}).
\end{equation}

\begin{example} Let $E$ be a finite set and $\rho:2^E\ra\Z$ given by $\rho(X)=|X|$.
Then $\eta(X)=0$ and $\eta^*(E\bs X)=|E\bs X|$ for all $X\subseteq E$. Hence
$T(x,y)=W(x,y,t)=x^n$.\end{example}

\begin{theorem} {\rm(MacWilliams identity)} Let $M=(E,\rho)$ be a combinatroid. Then
\[W_{M^*}(x,y,t)=t^{-\eta(M)}\,W_M(x+(t-1)y,x-y,t).\]\end{theorem}

\proof
\begin{eqnarray*}
W_M(x+(t-1)y,x-y,t)&=&(t y)^{\eta(E)}(x-y)^{\rho(E)}\,T_M((x+(t-1)y)/(x-y),x/y)\\
&=&t^{\eta(E)}\left[y^{\eta(E)}(x-y)^{\rho(E)}(x-y)^{-\eta^*(E)}y^{-\rho^*(E)}\right]\\
&&\times\,(x-y)^{\eta^*(E)}y^{\rho^*(E)}\,T_M((x+(t-1)y)/(x-y),x/y)\\
&=&t^{\eta(E)}(1)(x-y)^{\eta^*(E)}y^{\rho^*(E)}\,T_M((x+(t-1)y)/(x-y),x/y)\\
&=&t^{\eta(E)}(x-y)^{\eta^*(E)}y^{\rho^*(E)}\,T_{M^*}(x/y,(x+(t-1)y)/(x-y))\\
&=&t^{\eta(E)}W_{M^*}(x,y,t).\quad\qed
\end{eqnarray*}

We define the \emph{Whitney generating function}
\[ f(M;x,y):=\sum_{A\subseteq E}x^{\eta^*(E\bs A)}y^{\eta(A)}.\]
\begin{theorem} {\rm(Brylawski, Gordon; see~\cite{Gordon-z})} Let $M=(E,\rho)$ be a
combinatroid. Then
\begin{itemize}
 \item[(1)] Duality\,:
 \[f(M^*;x,y)=f(M;y,x).\]
 \item[(2)] Deletion-Contraction\,: For any $p\in E$,
 \[f(M;x,y)=x^{\eta^*(p)}f(M\bs p;x,y)+y^{1-\rho(p)}f(M/p;x,y).\]
\end{itemize}\end{theorem}

We now proceed to prove a deletion-contraction formula for the Tutte and Hamming
polynomials.
\begin{lemma}\label{reduction} Let $M=(E,\rho)$ be a combinatroid. Then
\begin{itemize}
 \item[(a)] \[(x-y)^{\eta(E)}y^{\rho(E)}(x/y-1)^{\rho(E)-\rho(E\bs p)}=(x-y)(x-y)^{|E\bs p|-\rho(E\bs p)}y^{\rho(E\bs p)}.\]
 \item[(b)] \[(x-y)^{\eta(E)}y^{\rho(E)}(x-y)^{\rho(p)-1}=y^{\rho(p)}(x-y)^{|E\bs p|-\rho(E)+\rho(p)}y^{\rho(E)-\rho(p)}.\]
\end{itemize}
\end{lemma}
\proof (a):
\begin{eqnarray*}
(x-y)^{\eta(E)}y^{\rho(E)}(x/y-1)^{\rho(E)-\rho(E\bs p)}
&=&(x/y-1)^{\rho(E)-\rho(E\bs p)}(x-y)^{\eta(E)}y^{\rho(E)}\\
&=&(x-y)^{\rho(E)-\rho(E\bs p)}y^{-\rho(E)+\rho(E\bs
p)}(x-y)^{\eta(E)}y^{\rho(E)}\\
&=&(x-y)^{\rho(E)-\rho(E\bs p)}y^{\rho(E\bs p)}(x-y)^{\eta(E)}\\
&=&(x-y)^{|E|-\rho(E\bs p)}y^{\rho(E\bs p)}\\
&=&(x-y)(x-y)^{|E\bs p|-\rho(E\bs p)}y^{\rho(E\bs p)}
\end{eqnarray*}
(b):
\begin{eqnarray*}
(x-y)^{\eta(E)}y^{\rho(E)}
&=&(x-y)^{1-\rho(p)}(x-y)^{|E\bs p|-\rho(E)+\rho(p)}y^{\rho(E)-\rho(p)}y^{\rho(p)}\\
&=&(x-y)^{1-\rho(p)}y^{\rho(p)}(x-y)^{|E\bs
p|-\rho(E)+\rho(p)}y^{\rho(E)-\rho(p)}. \qed
\end{eqnarray*}
From the Brylawski recurrence it follows:
\begin{proposition} Let $M=(E,\rho)$ be a combinatroid. Then
\[T_M(x,y)=(x-1)^{\eta^*(p)}\,T_{M\bs p}(x,y)+(y-1)^{1-\rho(p)}\,T_{M/p}(x,y).\]
\end{proposition}

\begin{example}\label{Tutte-recurrence}
Let $M$ be the demimatroid in Example~\ref{Contra-full}, and take $p=3$.
\[\begin{array}{|r|c|c|c|c|c|c|c|c|}\hline
 X&\ \emptyset&\ 1&\ 2&\ 3&12&13&23&E\\\hline\hline
 \rho&0&0&0&0&1&1&1&2\\\hline
 \rho^*&0&0&0&0&0&0&0&1\\\hline
 \rho^\circ&0&1&1&1&1&1&1&1\\\hline
 \rho^\cd&0&1&1&1&2&2&2&2\\\hline
 \end{array}.\]
\[T_M(x,y)=x - 2 x^2 + y - 3 x y + 3 x^2 y.\]
Set $\alpha:=\rho_{M\bs p}$ and $\beta:=\rho_{M/p}$.
\[\begin{array}{|r|c|c|c|c|}\hline
 X&\ \emptyset&\ 1&\ 2&\ 12\\\hline\hline
 \alpha&0&0&0&1\\\hline
 \alpha^*&0&0&0&1\\\hline
 \alpha^\circ&0&1&1&1\\\hline
 \alpha^\cd&0&1&1&1\\\hline
 \end{array}\,,\quad
 \begin{array}{|r|c|c|c|c|}\hline
 X&\ \emptyset&\ 1&\ 2&\ 12\\\hline\hline
 \beta&0&1&1&2\\\hline
 \beta^*&0&0&0&0\\\hline
 \beta^\circ&0&0&0&0\\\hline
 \beta^\cd&0&1&1&2\\\hline
 \end{array}\]
\[T_{M\bs p}(x,y)=-x-y+2xy;\quad T_{M/p}(x,y)=x^2.\]
\[(x-1)T_{M\bs p}(x,y)+(y-1)T_{M/p}(x,y)=T_M(x,y).\]
 \end{example}

From this we obtain the following recurrence for the Hamming polynomial.
\begin{theorem}\label{W-recurrence}
\[W_M(x,y,t)=(x-y)\,W_{M\bs p}(x,y,t))+t^{1-\rho(p)}y\,W_{M/p}(x,y,t).\]
\end{theorem}

\proof \begin{eqnarray*}
W_M(x,y,t)&=&(x-y)^{\eta(E)}y^{\rho(E)}\,T_M(x/y,(x+(t-1)y)/(x-y))\\
&=&(x-y)^{\eta(E)}y^{\rho(E)}[(x/y-1)^{\rho(E)-\rho(E\bs p)}\,T_{M\bs p}(x/y,(x+(t-1)y)/(x-y))\\
&&+((x+(t-1)y)/(x-y)-1)^{1-\rho(p)}\,T_{M/p}(x/y,(x+(t-1)y)/(x-y))]\\
&=&(x-y)(x-y)^{|E\bs p|-\rho(E\bs p)}y^{\rho(E\bs p)}\,T_{M\bs p}(x/y,(x+(t-1)y)/(x-y))\\
&&+(ty)^{1-\rho(p)}y^{\rho(p)}\\
&&\ \times\,(x-y)^{|E\bs p|-\rho(E)+\rho(p)}y^{\rho(E)-\rho(p)}\,T_{M/p}(x/y,(x+(t-1)y)/(x-y))\\
(\text{by }\ref{reduction})&=&(x-y)\,W_{M\bs
p}(x,y)+t^{1-\rho(p)}y\,W_{M/p}(x,y). \qed
\end{eqnarray*}

\begin{example} We continuous Example~\ref{Tutte-recurrence}.
\[W_M(x,y,t)=x^3 + 3 (t-1) x^2 y + 3 (1 - t) x y^2 + (t-1) y^3.\]
\[W_{M\bs p}(x,y,t)=x^2 + 2 (t-1) x y +(1 - t)y^2;\quad
W_{M/p}(x,y,t)=x^2.\]
\[(x-y)W_{M\bs p}(x,y,t)+t\, y\, W_{M/p}(x,y,t)=W_M(x,y,t).\]
\end{example}

\section{Extended Hamming polynomials}

For a combinatroid $M=(E,\rho)$ we define its \emph{characteristic polynomial}
as
\[p(M;t):=\sum_{X\subseteq E}(-1)^{|X|}t^{\rho(E)-\rho(X)}=(-1)^{\rho(E)}T_M(1-t,0)=
\sum_{X\subseteq E}(-1)^{|E\bs X|}t^{\eta^*(X)}.\] Thus the characteristic
polynomial of $M^*$ is
\[p(M^*;t)=\sum_{X\subseteq E}(-1)^{|E\bs X|}t^{\eta(X)}.\]
We generalize $p(M^*;t)$ for every $\sigma\subseteq E$ as:
$P_{M,\emptyset}(t):=1$ and
\begin{equation}\label{coefftr}
P_{M,\sigma}(t):=\sum_{\gamma\subseteq\sigma}(-1)^{|\sigma\bs\gamma|}t^{\eta(\gamma)}.
\end{equation}
We define the \emph{$j$-th generalized polynomial} $P_{M,j}(t)$ as
$P_{M,0}(t):=1$ and
\begin{equation}\label{Sum-coefftr}
P_{M,j}(t):=\sum_{|\sigma|=j}P_{M,\sigma}(t)=\sum_{|\sigma|=j}\,
\sum_{\gamma\subseteq\sigma}(-1)^{|\sigma\bs\gamma|}t^{\eta(\gamma)},\quad 1\leq
j\leq n.\end{equation} Identically as for matroids~\cite{JohRokVer-z}, we define
the \emph{Hamming polynomial} of a combinatroid $M$ by
 \begin{equation}\label{WPMj} W_M(x,y,t):=\sum_{j=0}^n P_{M,j}(t)x^{n-j}y^j.\end{equation}
Next, following~\cite{JohRokVer-z}, we will verify that this definition
coincides with the one given in Eq.~(\ref{WHP-definition}).

\begin{lemma} \[W_M(x,y,t)=\sum_{\sigma\subseteq E}(x-y)^{|E|-|\sigma|}y^{|\sigma|}\
t^{\eta(\sigma)}.\]
\end{lemma}

\proof Set $n=|E|$.
\begin{eqnarray*}
\sum_{\sigma}(x-y)^{n-|\sigma|}y^{|\sigma|} t^{\eta(\sigma)}
&=&\sum_{\sigma}\sum_{i=0}^{n-|\sigma|}{n-|\sigma|\choose
i}x^{i}(-y)^{n-|\sigma|-i}y^{|\sigma|}\ t^{\eta(\sigma)}\\
&=&\sum_{\sigma}\sum_{\gamma\subseteq E\bs\sigma}
x^{|\gamma|}y^{n-|\gamma|}(-1)^{n-|\sigma|-|\gamma|}\ t^{\eta(\sigma)}\\
&=&\sum_{\gamma}x^{|\gamma|}y^{n-|\gamma|}
\sum_{\sigma\subseteq E\bs\gamma}(-1)^{n-|\gamma|-|\sigma|}\ t^{\eta(\sigma)}\\
&=&\sum_{\gamma}x^{|\gamma|}y^{n-|\gamma|}P_{M,E\bs\gamma}(t)\\
&=&\sum_{\gamma}x^{n-|\gamma|}y^{|\gamma|}P_{M,\gamma}(t)\\
&=&W_M(x,y,t).\quad\qed
\end{eqnarray*}

\begin{theorem}\label{W-Tutte-relation}
\[W_M(x,y,t)=(x-y)^{\eta(E)}y^{\rho(E)}\,T_M(\frac{x}{y},\frac{x+(t-1)y}{x-y}).\]
\end{theorem}

\proof
\begin{eqnarray*}
T_M(\frac{x}{y},\frac{x+(t-1)y}{x-y})&=&\sum_{\sigma}(\frac{x}{y}-1)^{\eta^*(E\bs\sigma)}
(\frac{x+(t-1)y}{x-y}-1)^{\eta(\sigma)}\\
&=&\sum_{\sigma}\frac{(x-y)^{\eta^*(E\bs\sigma)}}{y^{\eta^*(E\bs\sigma)}}\
\frac{(t y)^{\eta(\sigma)}}{(x-y)^{\eta(\sigma)}}\\
&=&\sum_{\sigma}\frac{(x-y)^{\eta^*(E\bs\sigma)-\eta(\sigma)}}
{y^{\eta^*(E\bs\sigma)-\eta(\sigma)}}\ t^{\eta(\sigma)}\\
&=&\sum_{\sigma}\frac{(x-y)^{\rho(E)-|\sigma|}}
{y^{\rho(E)-|\sigma|}}\ t^{\eta(\sigma)}\\
&=&\frac{(x-y)^{\rho(E)-n}}{y^{\rho(E)}}
\sum_{\sigma}(x-y)^{n-|\sigma|}y^{|\sigma|}\ t^{\eta(\sigma)}\\
&=&\frac{(x-y)^{\rho(E)-n}}{y^{\rho(E)}}W_M(x,y,t).\quad\qed
\end{eqnarray*}

\begin{theorem}\label{theorem-of-equivalence}
\[T_M(x,y)=(x-1)^{-\eta(E)}x^{|E|}\,W_M(1,x^{-1},(x-1)(y-1)).\]
\end{theorem}

\proof A straightforward evaluation shows that\begin{eqnarray*}
 W_M(1,x^{-1},(x-1)(y-1))&=&(1-x^{-1})^{n-\rho(E)}x^{-\rho(E)}T_M(x,y)\\
 &=&(x-1)^{n-\rho(E)}x^{-n}T_M(x,y).\quad\qed
 \end{eqnarray*}

\begin{example} Let $\D$ be a simplicial complex of dimension $d$; so $d+1$
is the largest cardinality of a face. The $f$-{\it po\-ly\-no\-mial} of $\D$ is
defined as \[f(\D,t):=t^{d+1}+c_1 t^{d-1}+\cdots+c_d,\] where $c_i$ is the
number of faces of cardinality $i$, and its $h$-{\it polynomial} is defined as
$h(\D,t):=f(\D,t-1)$. It is well known that $f(\D,t)=T(t+1,1)$, where $T(x,y)$
is the Tutte polynomial of $\D$. Thus, by Theorem~\ref{theorem-of-equivalence},
\[f(\D,t)=(x+1)^{|E|}x^{-\eta(E)}W(1,(x+1)^{-1},0).\]
For instance, let $\D$ be the simplicial complex with facets $12,234,345$, i.e.
\[\D=\{\emptyset, 1, 2, 3, 4, 5, 12, 23, 24, 34, 35, 45, 234, 345\}.\]
\[T_{\D^\uparrow}(x,y)=x - 2 x^2 + x^3 + y - 4 x y + 4 x^2 y - y^2 + 2 x y^2.\]
\[W_{\D^\uparrow}(x,y,t)=x^5 + 4 (t-1) x^3 y^2 +
    4 (1-t) x^2 y^3 + (-1 - t + 2 t^2) x y^4 + (1-t) t y^5.\]
Thus, the $f$-polynomial of $\D$ is
    \[(x+1)^5x^{-2}W_{\D^\uparrow}(1,(x+1)^{-1},0)=x^3+5x^2+6x+2.\]
\end{example}

Let $M=(E,\rho)$ be a nontrivial demimatroid, $P_{M,j}(t)$ the polynomial
defined in Eq.~(\ref{Sum-coefftr}), and $\delta$ the minimum $j>0$ such that
$P_{M,j}(t)\neq 0$.

\begin{proposition} $\delta=\d_1(M^\circ)$ and $P_{M,\delta}(t)=c(t-1)$,
where \[c=|\{X\subseteq E: |X|=\d_1(M^\circ)\}|.\]\end{proposition}

\begin{proof} Fix $X\subseteq E$ such that $\eta(X)=1$ and $|X|=\d_1(M^\circ)$.
If $\sigma\subseteq E$ and $|\sigma|<|X|$, then by Lemma~\ref{Wei-number}(i),
applied to the restriction of $\eta$ to $\sigma$, it holds that
$\eta(\sigma)=0$. Thus
$0=P_{M,\sigma}(t):=\sum_{\gamma\subseteq\sigma}(-1)^{|\sigma\bs\gamma|}t^{\eta(\gamma)}$.
The same result holds if $|\sigma|=|X|$ and $\eta(\sigma)=0$. On the other hand,
$P_{M,X}(t)=t-1$. Therefore, we obtain the desired result.\end{proof}

We call the number $\d_1(M^\circ)$ the \emph{formal minimum distance} of $M$.

\begin{proposition}\label{Tutte-of-um} Let $M=(E,\rho)$ be a uniform matroid, with
$n=|E|$ and $k=\rho(E)$. Then
\[T_M(x ,y)=\sum_{i=0}^{k-1}{n\choose i}(x-1)^{k-i}+{n\choose k}+
\sum_{i=k+1}^{n}{n\choose i}(y-1)^{i-k}.\]
\end{proposition}

\begin{proof} $\rho(X)=|X|$ if $|X|\leq k$ and $\rho(X)=k$ if $|X|>k$. Hence,
$\eta(X)=0$ if $|X|\leq k$ and $\eta(X)=r$ if $|X|=k+r$ with $r>0$. Moreover,
$\eta^*(X)=0$ if $|E\bs X|\geq k$, i.e. $|X|\leq n-k$, and $\eta^*(X)=r$ if
$|E\bs X|=k-r$ with $r>0$, i.e. $|X|=n-k+r$.\end{proof}

Let $M=(E,\rho)$ be a demimatroid. Set $n=|E|$ and write
\[W_M(x,y,t)=x^n+\sum_{j=\delta}^n A_j(t)x^{n-j}y^j.\]
where $\delta$ is the formal minimum distance of $M$.

\begin{proposition} Let $M=(E,\rho)$ be a uniform matroid, with
$n=|E|$, $k=\rho(E)$ and $\delta=\d_1(M^\circ)$. Then, for $\delta\leq i\leq n$,
\[A_i(t)=(t-1){n\choose i}\sum_{j=0}^{i-\delta}(-1)^j{i-1\choose j}t^{i-\delta-j}.\]
\end{proposition}

\begin{proof} The proof readily follows from
Proposition~\ref{Tutte-of-um}.\end{proof}

\begin{example} Let $M^\circ=(E,\rho^\circ)$ be as in Example~\ref{Contra-full}.
$M^\circ$ is a uniform matroid of rank $1$. $T_{M^\circ}(x,y)= x+y+y^2$,
$W_{M^\circ}(x,y,t)=x^3 + 3 (t-1) x y^2 + (2 - 3 t + t^2) y^3$, $\delta=2$,
$A_1(t)=0$, $A_2(t)=3(t-1)$, $A_3(t)=2-3t+t^2$.\end{example}

\begin{example} Let $M^\circ=(E=\{1,2,3,4\},\rho^\circ)$ be the uniform
matroid given by:
\[\begin{array}{|r|c|c|c|c|c|c|c|c|c|c|c|c|c|c|c|c|}\hline
 X&\emptyset&1&2&3&4&12&13&14&23&24&34&123&124&134&234&E\\\hline\hline
 \rho&0&0&0&0&0&0&0&0&0&0&0&1&1&1&1&2\\\hline
 \rho^*&0&0&0&0&0&0&0&0&0&0&0&1&1&1&1&2\\\hline
 \rho^\circ&0&1&1&1&1&2&2&2&2&2&2&2&2&2&2&2\\\hline
 \rho^\cd&0&1&1&1&1&2&2&2&2&2&2&2&2&2&2&2\\\hline
 \end{array}.\]
 $M^\circ$ is a uniform matroid of rank $2$. $T_{M^\circ}(x,y)=2x+x^2+2y+y^2$,
 $W_{M^\circ}(x,y)=x^4 + 4 (t-1) x y^3 + (3 - 4 t + t^2) y^4$, $\delta=3$,
 $A_1(t)=0$, $A_2=0$, $A_3(t)=4(t-1)$, $A_4(t)=3-4t+t^2$.
\end{example}

\section{Elongations}

Let $M=(E,\rho)$ be a demimatroid with nullity function $\eta$. For $0\leq i\leq
\eta(M)$ we define the \emph{$i$-th elongation} of $M$ as the demimatroid
$M^{[i]}:=(E,\rho^{[i]})$, where
\[\rho^{[i]}(\sigma):=\min\{|\sigma|,\rho(\sigma)+i\},\]
or equivalently,
\[\rho^{[i]}(\sigma)=\begin{cases}
|\sigma|,& \eta(\sigma)\leq i\\
 \rho(\sigma)+i,& \eta(\sigma)>i.
\end{cases}\]
Note that $\rho^{[0]}\equiv\rho$, $\rho^{[i]}\equiv(\rho^{[1]})^{[i-1]}$ and
$\rho^{[\eta(M)]}(\sigma)=|\sigma|$ for all $\sigma\subseteq E$. When there is
no confusion, we will write $M[i]$ instead of $M^{[i]}$.

\begin{proposition} Let $M=(E,\rho)$ be a demimatroid. Then $M[i]$, as defined
above, is a demimatroid.\end{proposition}

\proof Obviously $\rho^{[i]}(\emptyset)=0$. Let $X\subseteq E$ and $x\in E$.

If $x\in X$, obviously $\rho^{[i]}(X)\leq\rho^{[i]}(X\cup
x)\leq\rho^{[i]}(X)+1$, so we may assume $x\not\in X$.

If $\rho^{[i]}(X\cup x)=|X|+1$, thus $\rho^{[i]}(X)=|X|$ and
$\rho^{[i]}(X)\leq\rho^{[i]}(X\cup x)\leq\rho^{[i]}(X)+1$.

If $\rho^{[i]}(X\cup x)=\rho(X)+i$, thus $\rho^{[i]}(X)=\rho(X)+i$ and
$\rho^{[i]}(X)\leq\rho^{[i]}(X\cup x)\leq\rho^{[i]}(X)+1$.\quad\qed

\medskip Since $1\leq i\leq \eta(M)=|E|-\rho(E)$, it holds that
$\rho(E)+i\leq|E|$, so $\rho^{[i]}(M[i])=\rho(M)+i$. If $X\subseteq E$, then the
rank function of $M|_X$ is the restriction of $\rho$ to $X$. We point out that
from this it follows that $(M[i])|_X=(M|_X)[i]$.

\bigskip The nullity function of $M[i]$ is given by
\[\eta^{[i]}(\sigma)=\max\{0,\eta(\sigma)-i\},\]
or equivalently,
\[\eta^{[i]}(\sigma)=\begin{cases}
0,& \eta(\sigma)\leq i\\
 \eta(\sigma)-i,& \eta(\sigma)>i.
\end{cases}\]
An easy verification shows that
\begin{equation}\label{etai0} \eta^{[i]}(\sigma)=0\quad\text{if and only if}\quad
\eta(\sigma)\leq i.\end{equation}

\begin{proposition} Let $M=(E,\rho)$ be a demimatroid. Then
$\d_{r+1}(M^\circ)=\d_1(M[r]^\circ)$.\end{proposition}

\begin{proof} Choose $X\subseteq E$ such that $\eta(X)=r+1$ and
$|X|=\d_{r+1}(M^\circ)$. Hence $\eta^{[r]}(X)=\max\{0,\eta(X)-r\}=1$, and so
$\d_1(M[r]^\circ)\leq|X|=\d_{r+1}(M^\circ)$. Similarly, choose $Y\subseteq E$
such that $\eta^{[r]}(Y)=1$ and $|Y|=\d_1(M[r]^\circ)$. Hence $\eta(Y)=r+1$, and
so $\d_{r+1}(M^\circ)\leq|Y|=\d_1(M[r]^\circ)$.\end{proof}

\section{Betti numbers}

Let $R=K[x_1,\ldots,x_n]$ be a polynomial ring over the field $K$. We consider
$R$ provided with its canonical $\Z$-grading. Associated with each homogeneous
ideal $I$ of $R$ there is a \emph{minimal graded free resolution}
\[
 0 \ra \mathsmaller{\bigoplus}_j R(-j)^{\beta_{pj}} \ra \cdots \ra
 \mathsmaller{\bigoplus}_j R(-j)^{\beta_{1j}}\ra R \ra R/I\ra 0,\]
where $R(-j)$ denotes the $R$-module obtained by shifting the degrees of $R$ by
$j$, i.e $R(-j)_a=R_{a-j}$. The number $\beta_{ij}$ in the resolution may be
interpreted as the minimum number of generators of degree $j$ in the $i$-th
sizygie of $R/I$; or equivalently \[\beta_{ij}(R/I):=\beta_{ij} = \dim
\text{Tor}_i(R/I,K)_j.\] These $\beta_{ij}$'s are called the \emph{graded Betti
numbers} of $R/I$. We collect all they together by defining the {\it graded
Betti polynomial} of $R/I$ as
\[B(R/I;x,y):=\sum_{i=0}^p\sum_j\beta_{ij}x^i y^{j}.\]
\begin{example} Let $I\subset R=K[x_1,\ldots,x_5]$ be the monomial ideal
\[I=\langle x_1x_2,x_2x_3,x_3x_4,x_4x_5\rangle.\] We have the resolution
\[0 \ra R(-5) \ra R^3(-3)\oplus R(-4) \ra R^4(-2) \ra R \ra R/I \ra 0,\]
so that
\[B(R/I;x,y)=1+4xy^2+3x^2y^3+x^2y^4+x^3y^5.\]
\end{example}

Let $\D$ be a simplicial complex; we assume that all the vertices belongs to
$\D$. It is convenient, abusing notation, to identify $\sigma\subseteq[n]$ with
the characteristic vector $\sigma=(\sigma_i)\in\{0,1\}^n$ such that $\sigma_i =
1$ if $i \in \sigma$; and write $|\sigma|:=\sigma_1+\cdots+\sigma_n$. For
$\sigma\subseteq[n]$ we denote by $\D_\sigma$ the simplicial complex that
results from the restriction of $\D$ to the vertex set $\sigma$.

Given a simplicial complex $\D$, let $I_\D$ denote its \emph{Stanley-Reisner
ideal} and $K[\D]$ its \emph{Stanley-Reisner ring}, i.e. $I_\D=\langle
x_{i_1}\cdots x_{i_r}: \{i_1,\ldots,i_r\}\notin\D\rangle\subset R$ and
$K[\D]=R/I_\D$. Let's also denote by $\wt{H}_i(\D;K)$ the $i$-th \emph{reduced
homology group} of $\D$ with coefficients in the field $K$. We have the
fundamental result:
\begin{theorem}{\rm(Hochster's Formula \cite{H-z})}
Let $\D$ be a simplicial complex with vertex set $[n]$. Then
\[\beta_{ij}(R/I_\D) = \sum_{\sigma\subseteq[n];\,|\sigma|=j}\dim \wt{H}_{j-i-1}( \D_\sigma).\]
\end{theorem}
If, instead of the $\Z$-grading, we consider $R$ provided with its
$\Z^n$-grading, and for any $\sigma\subseteq[n]$ we define
$\beta_{i\sigma}(R/I):=\dim \text{Tor}_i(R/I,K)_\sigma$, then we have
\begin{theorem}\label{MHF} {\rm(Multigraded Hochster's Formula)}
Let $\D$ be a simplicial complex with vertex set $[n]$. For any
$\sigma\subseteq[n]$ we have that \[\beta_{i\sigma}(R/I_\D) = \dim
\wt{H}_{|\sigma|-i-1}( \D_\sigma).\]\end{theorem}

\section{Hamming polynomial vs Betti numbers}

Let $\D$ be a simplicial complex of dimension $d$ and denote by $f_i$ the number
of $i$-dimensional faces of $\D$. The \emph{reduced Euler characteristic} of
$\D$ is defined as
\[\wt{\chi}(\D):=\sum_{i=-1}^d(-1)^i\dim\wt{H}_i(\D;K).\]

\begin{lemma}\label{EPF} {\rm(Euler-Poincar\'e formula)}
The reduced Euler characteristic of a simplicial complex does not depend of the
field and
\[\wt{\chi}(\D)=-1+f_0-\cdots+(-1)^{d}f_{d}.\]\end{lemma}

Let $M=(E,\rho)$ be a demimatroid with nullity function $\eta$, and let $M[i]$
be its $i$-th elongation. Set $n=|E|$ and denote by $I_{M[i]}$ the
Stanley-Reisner ideal of $M[i]$, viewed as a simplicial complex.

\begin{lemma}
For $\sigma\subseteq E$ the coefficient of $t^r$ in $P_{M,\sigma}(t)$ is equal
to
\[\sum_{i=0}^n(-1)^i\left(\beta_{i\sigma}(R/I_{M[r]})-\beta_{i\sigma}(R/I_{M[r-1]})\right).\]
\end{lemma}

\proof According to Eq.~(\ref{coefftr}), the coefficient of $t^r$ is
\[s_{r\sigma}=(-1)^{|\sigma|}\sum_{\gamma\subseteq\sigma;\,\eta(\gamma)=r}(-1)^{|\gamma|}.\]
From Eq.~(\ref{etai0}) we have
 \begin{eqnarray*}
 s_{r\sigma}
 &=&(-1)^{|\sigma|}\left(\sum_{\gamma\subseteq\sigma;\,\eta^{[r]}(\gamma)=0}(-1)^{|\gamma|}-
\sum_{\gamma\subseteq\sigma;\,\eta^{[r-1]}(\gamma)=0}(-1)^{|\gamma|}\right).
\end{eqnarray*}
By Eq.~\ref{etai0} and Lemma~\ref{EPF},
\begin{eqnarray*}
(-1)^{|\sigma|}\sum_{\gamma\subseteq\sigma;\,\eta^{[r]}(\gamma)=0}(-1)^{|\gamma|}
&=&(-1)^{|\sigma|+1}\left(\sum_{\gamma\subseteq\sigma;\,\eta^{[r]}(\gamma)=0}(-1)^{|\gamma|-1}\right)\\
&=&(-1)^{|\sigma|+1}\left(\sum_{i=-1}^{\rho^{[r]}(\sigma)}(-1)^i\dim\wt{H}_i(M[r]_\sigma;K)\right)\\
&=&(-1)^{|\sigma|+1}\left(\sum_{j=\eta^{[r]}(\sigma)-1}^{|\sigma|}(-1)^{|\sigma|-j-1}\dim\wt{H}_{|\sigma|-j-1}(M[r]_\sigma;K)\right)\\
&=&\sum_{j=\eta^{[r]}(\sigma)-1}^{|\sigma|}(-1)^{j}\dim\wt{H}_{|\sigma|-j-1}(M[r]_\sigma;K)\\
(\text{by}~\ref{MHF})&=&\sum_{j=0}^{|\sigma|}(-1)^j\beta_{j\sigma}(R/I_{M[r]}).
\end{eqnarray*}
Similarly,
\[(-1)^{|\sigma|}\sum_{\gamma\subseteq\sigma;\,\eta^{[r-1]}(\gamma)=0}(-1)^{|\gamma|}
=\sum_{j=0}^{|\sigma|}(-1)^j\beta_{j\sigma}(R/I_{M[r-1]}).\quad\qed\]

\begin{corollary}\label{bettimainlemma} For each $1\leq j\leq n$ the coefficient of
$t^r$ in $P_{M,j}(t)$ is equal to
\[\sum_{i=0}^n(-1)^i\left(\beta_{ij}(R/I_{M[r]})-\beta_{ij}(R/I_{M[r-1]})\right).\]
\end{corollary}

\proof Recall that $P_{M,j}(t)=\sum_{|\sigma|=j}P_{M,\sigma}(t)$ and
$\beta_{ij}(R/I_{M[r]})=\sum_{|\sigma|=j}\beta_{i\sigma}(R/I_{M[r]})$. Hence the
coefficient of $t^r$ in $P_{M,j}$ is
\[
\sum_{i=0}^n(-1)^i\left(\beta_{ij}(R/I_{M[r]})-\beta_{ij}(R/I_{M[r-1]})\right).
\qed
\]

\begin{theorem}\label{W-Betti-relation}
\[W(x,y,t)=x^n\sum_{r=0}^\eta\left(B_{M[r]}(-1,y/x)-B_{M[r-1]}(-1,y/x)\right)t^r.\]
\end{theorem}

\proof By definition $W(x,y,t)=\sum_{j=0}^n P_{M,j}(t)x^{n-j}y^j$. By
Corollary~\ref{bettimainlemma},
\begin{eqnarray*}
W(x,y,t)&=&\sum_{j=0}^n\left(\sum_{r=0}^n\left(\sum_{i=0}^n(-1)^i
\left(\beta_{ij}(R/I_{M[r]})-\beta_{ij}(R/I_{M[r-1]})\right)\right)t^r\right)x^{n-j}y^j\\
&=&\sum_{r=0}^n\left(\sum_{j=0}^n\left(\sum_{i=0}^n(-1)^i
\left(\beta_{ij}(R/I_{M[r]})-\beta_{ij}(R/I_{M[r-1]})\right)\right)x^{n-j}y^j\right)t^r\\
&=&\sum_{r=0}^n\left(\sum_{i=0}^n(-1)^i\left(\sum_{j=0}^n
\left(\beta_{ij}(R/I_{M[r]})-\beta_{ij}(R/I_{M[r-1]})\right)x^{n-j}y^j\right)\right)t^r\\
&=&x^n\sum_{r=0}^n\left(\sum_{i=0}^n\left(\sum_{j=0}^n
\left(\beta_{ij}(R/I_{M[r]})-\beta_{ij}(R/I_{M[r-1]})\right)(-1)^i(y/x)^j\right)\right)t^r\\
&=&x^n\sum_{r=0}^{\eta(E)}\left(B_{M[r]}(-1,y/x)-B_{M[r-1]}(-1,y/x)\right)t^r.\quad\qed
\end{eqnarray*}

\begin{remark}
\begin{itemize}
 \item[(i)] $B_{M[-1]}(x,y)=0$ and  $B_{M[\eta(E)]}(x,y)=1$.
 \item[(ii)] $W(x,y,0)=x^n B_M(-1,y/x).$\end{itemize}\end{remark}

\section{Examples}

\begin{example}\label{casi-wheel} Let $G$ be the graph
  \[
 \xygraph{
 !{<0cm,0cm>:<0cm,2cm>:<-2cm,0cm>}
 !{(0,0);a(0)**{}?(0)}*{\bullet}@\cir{}="a1"
 !{(0,0);a(0)**{}?(1)}*{\bullet}@\circ{}="a4"
 !{(0,0);a(72)**{}?(1)}*{\bullet}@\cir{}="a3"
 !{(0,0);a(144)**{}?(1)}*{\bullet}@\cir{}="a2"
 !{(0,0);a(216)**{}?(1)}*{\bullet}@\cir{}="a6"
 !{(0,0);a(288)**{}?(1)}*{\bullet}@\cir{}="a5"
 "a1"-"a2" "a2"-"a3" "a3"-"a4" "a4"-"a5" "a5"-"a6"
 "a1"-"a2" "a1"-"a3" "a1"-"a4" "a1"-"a5" "a1"-"a6"
 "a1"!{+U*++!U{_1}} "a2"!{+U*++!RU{_2}}
 "a3"!{+U*++!R{_3}} "a4"!{+D*++!D{_4}} "a5"!{+U*++!L{_5}}
 "a6"!{+U*++!LU{_6}}
 }\]
and let $\D$ be the simplicial complex whose facets are the nine edges of this
graph. Let us consider $\D$ provided with its natural structure of demimatroid,
i.e. $\rho(\sigma)=\max\{|X|: X\subseteq\sigma \text{ and } X\in\D\}$. The
circuits of $\D$ are $\{24, 25, 26, 35, 36, 46, 123, 134, 145, 156\}$; here we
have written $24$ instead of $\{2,4\}$, and so on. We have
\[T(x,y)=-x + x^2 - y + 4 x y + 2 y^2 + x y^2 + 2 y^3 + y^4\]
and
\begin{eqnarray*}
 W(x,y,t)&=&(x-y)^4y^2\ T(\frac{x}{y},\frac{x+(t-1)y}{x-y})\\
   &=&x^6 + 6(-1 + t) x^4 y^2 + (4 - 5 t + t^2) x^3 y^3\\
    && +\, 3(3 - 7 t + 4 t^2) x^2 y^4 + 3(-4 + 11 t - 9 t^2 + 2 t^3) x y^5\\
    && +\, (4 - 13 t + 14 t^2 - 6 t^3 + t^4) y^6.
\end{eqnarray*}
The Betti polynomial of the elongations of $\D$, over $\Q$, are
\begin{eqnarray*}
B_0(x,y) &=&1 + 6x y^2+ 4x y^3 + 8x^2 y^3+ 12x^2 y^4+3x^3 y^4+12x^3y^5 +4 x^4y^6;\\
B_1(x,y) &=& 1 + x y^3+12x y^4+21x^2 y^5+9x^3y^6;\\
B_2(x,y) &=& 1+6x y^5+5x^2y^6;\\
B_3(x,y) &=& 1+x y^6;\\
B_4(x,y) &=& 1.
\end{eqnarray*}
From this we obtain
\begin{eqnarray*}x^6\sum_{r=0}^4\left(B_{M[r]}(-1,y/x)-B_{M[r-1]}(-1,y/x)\right)t^r&=&x^6 +
    6(-1 + t) x^4 y^2\hspace{-4cm}\\
    &&\hspace{-4cm} +\, (4 - 5 t + t^2) x^3 y^3 + 3(3 - 7 t + 4 t^2) x^2 y^4\\
    &&\hspace{-4cm} +\, 3(-4 + 11 t - 9 t^2 + 2 t^3) x y^5\\
    &&\hspace{-4cm} +\, (4 - 13 t + 14 t^2 - 6 t^3 + t^4) y^6.
\end{eqnarray*}
\end{example}

\begin{example} Let $\D$ be the simplicial complex whose faces are the
independent vertex sets of the graph $G$ in Example~\ref{casi-wheel}, i.e. the
facets of $\D$ are $\{1, 25, 35, 36, 246\}$. The circuits of $\D$ are all the
edges of $G$. We have
\[T(x,y)=x - 2 x^2 + x^3 + y - 2 x y + x^2 y + y^2 - 5 x y^2 +
    4 x^2 y^2 - 2 y^3 + 3 x y^3\]
    and
\begin{eqnarray*}
W(x,y,t)&=&(x-y)^3y^3\ T(\frac{x}{y},\frac{x+(t-1)y}{x-y})\\
&=&x^6+ 9(-1 + t)x^4 y^2+ (17 - 21 t + 4 t^2) x^3 y^3\\
&&+\, 12 (-1 + t) x^2 y^4 + 3(1 + t - 3 t^2 + t^3) x y^5 + t(-3 + 5t - 2
t^2)y^6.
\end{eqnarray*}
The Betti polynomial of the elongations of $\D$ are
\begin{eqnarray*}
B_0(x, y) &=&
    1 + 9x y^2 + 17x^2y^3 + x^2y^4 + 13x^3y^4 + 2x^3y^5 + 5x^4y^5 + x^4y^6 +
      x^5y^6;\\
B_1(x, y) &=& 1 + 4x y^3 + 3x y^4 + 3x^2y^4 + 6x^2y^5 + 3x^3y^6;\\
B_2(x, y) &=& 1 + 3x y^5 + 2x^2y^6;\\
B_3(x, y) &=& 1.
\end{eqnarray*}
From this we obtain
\begin{eqnarray*}x^6\sum_{r=0}^3\left(B_{M[r]}(-1,y/x)-B_{M[r-1]}(-1,y/x)\right)t^r&=&x^6
+\ 9(-1 + t)x^4 y^2\\ &&\hspace{-4cm}+\, (17 - 21 t + 4 t^2) x^3 y^3 +\ 12 (-1 +
t) x^2 y^4\\ &&\hspace{-4cm} +\, 3(1 + t - 3 t^2 + t^3) x y^5 + t(-3 + 5t - 2
t^2)y^6.
\end{eqnarray*}
\end{example}

\begin{example} Let $C$ be the Hamming linear $[8,4,4]_2$ code, with parity check
matrix
\[H=\left(\begin{matrix}
1&0&0&0&0&1&1&1\\
0&1&0&0&1&0&1&1\\
0&0&1&0&1&1&0&1\\
0&0&0&1&1&1&1&0\\
\end{matrix}\right).\]
We have
\[T(x,y)=6 x + 10 x^2 + 4 x^3 + x^4 + 6 y + 14 x y + 10 y^2 + 4 y^3 +
    y^4\]
and
\begin{eqnarray*}
W(x,y,t)&=&x^8 + 14 (-1 + t) x^4 y^4 +
    28 (2 - 3 t + t^2) x^2 y^6 \\
    &&+\,8 (-8 + 14 t - 7 t^2 + t^3) x y^7 + (21 - 42 t +
          28 t^2 - 8 t^3 + t^4) y^8.
\end{eqnarray*}
The Betti polynomial of the elongations of $M[H]$ are
\begin{eqnarray*}
 B_0(x,y)&=&1+14xy^4+56x^2y^6+64x^3y^7+21x^4y^8;\\
 B_1(x,y)&=&1+28xy^6+48x^2y^7+21x^3y^8;\\
 B_2(x,y)&=&1+8xy^7+7x^2y^8;\\
 B_3(x,y)&=&1+xy^8;\\
 B_4(x,y)&=&1.
\end{eqnarray*}
From this we obtain
\[x^8\sum_{r=0}^4\left(B_{M[r]}(-1,y/x)-B_{M[r-1]}(-1,y/x)\right)t^r=W(x,y,t).\]
\end{example}

\begin{example} Let $\D$ be the simplicial complex whose facets
are the $2$-dimensional faces determined by the triangulation of the projective
plane

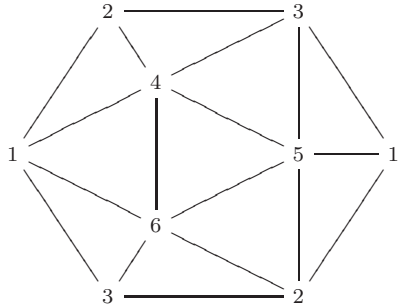
\begin{figure}[ht]
$$
\xygraph{
 !{<0cm,0cm>;<1.5cm,0cm>;<0cm,1.5cm>}
 !{(0,0)}*+{\small\text{$_3$}}="v1"
 !{(2,0)}*+{\small\text{$_2$}}="v2"
 !{(3,1.5)}*+{\small\text{$_1$}}="v3"
 !{(2,3)}*+{\small\text{$_3$}}="v4"
 !{(0,3)}*+{\small\text{$_2$}}="v5"
 !{(-1,1.5)}*+{\small\text{$_1$}}="v6"
 !{(.5,2.25)}*+{\small\text{$_4$}}="v7"
 !{(.5,.75)}*+{\small\text{$_6$}}="v8"
 !{(2,1.5)}*+{\small\text{$_5$}}="v9"
 "v1"-"v2" "v2"-"v3" "v3"-"v4" "v4"-"v5" "v5"-"v6" "v6"-"v1"
 "v2"-"v9" "v9"-"v4" "v4"-"v7" "v6"-"v7" "v2"-"v8" "v8"-"v6"
 "v5"-"v7" "v3"-"v9" "v1"-"v8"
 "v7"-"v8" "v7"-"v9" "v8"-"v9"
 }
$$
\end{figure}

i.e., the facets of $\D$ are $\{124,234,345,135,125,256,236,136,146,456\}.$

In characteristic $2$ the Betti polynomials are
\begin{eqnarray*}
B_0(x,y) &=& 1 + 10x y^3 + 15x^2y^4 + 6x^3y^5 + x^3y^6 + x^4y^6;\\
B_1(x,y) &=& 1 + 6x y^5 + 5x^2y^6;\\
B_2(x,y) &=& 1 + x y^6;\\
B_3(x,y) &=& 1.
\end{eqnarray*}
In characteristic $3$ the Betti polynomials are
\begin{eqnarray*}
B_0(x,y) &=& 1 + 10x y^3 + 15x^2y^4 + 6x^3y^5;\\
B_1(x,y) &=& 1 + 6x y^5 + 5x^2y^6;\\
B_2(x,y) &=& 1 + x y^6;\\
B_3(x,y) &=& 1.
\end{eqnarray*}
Even though these polynomials do depend of the characteristic of the field, in
both cases it results that
\[T(x,y)=-4 x + 3 x^2 + x^3 - 4 y + 10 x y + 3 y^2 + y^3\]
and
\[W(x,y,t)=x^6 + 10 (-1 + t) x^3 y^3 -
    15 (-1 + t)x^2 y^4 +
    6 (-1 + t^2) x y^5 +
    t (5 - 6 t + t^2) y^6.\]

Note that the coefficient of $x^2 y^4$, i.e. $-15(t-1)$, is negative for any
$t>1$, so $W(x,y,t)$ cannot be the weight enumerator of any code over a finite
field.

\bigskip The Duursma zeta polynomial corresponding to $W(x,y,t)$ is
\[P_q(t)=(1/2)(1+(1-q)t+qt^2).\]

This polynomial has negative discriminant for
$q\in(3-2\sqrt{2},3-2\sqrt{2})\approx(0.17,5.82)$. For $q$ in this interval, the
roots of $P_q(t)$ lie in the circle $(x+1)^2+y^2=2$, moreover all roots have
module $1/\sqrt{q}$, so that $P_q(t)$ satisfies the Riemann hypothesis.
See~\cite{Duursma-z}.
\end{example}

\begin{example} Let $M$ be the Vamos matroid, i.e. the ground set is $E=\{1,\ldots,8\}$
and the bases are all the subsets of $E$ of size $4$, except $\{1234, 2356,
1456, 2378, 1478\}$. We have
\[T(x,y)=x^4+4x^3+10x^2+15x+5xy+15y+10y^2+4y^3+y^4\]
and
\begin{eqnarray*}
W(x,y,t)&=&x^8 + 5(-1 + t) x^4 y^4 +
    36(-1 + t) x^3 y^5
     + 2(55 - 69 t + 14 t^2) x^2 y^6\\
    && +\, 4(-25 + 37 t - 14 t^2 + 2 t^3) x y^7
     + (30 - 51 t + 28 t^2 - 8 t^3 + t^4) y^8.
\end{eqnarray*}
\end{example}

\section{Generalized Hamming polynomial}

For positive integers $j\leq m$ and $q$ an indeterminate, let us define
\begin{eqnarray*}
 [m]_q&:=&1+q+\cdots+q^{m-1}\\
 {[m]_q}!&:=&[1]_q\,[2]_q\cdots[m]_q\\
 \left[\begin{array}{c} m\\ j\end{array}\right]_q
 &:=&\frac{[m]_q!}{[j]_q!\,[m-j]_q!}\\
 \langle m\rangle_q&:=&(q^m-1)(q^m-q)\cdots(q^m-q^{m-1}).
\end{eqnarray*}

Since \[\left[\begin{array}{c} m\\ j\end{array}\right]_q=
 \left[\begin{array}{c} m-1\\ j\end{array}\right]_q+
 q^{m-j}\left[\begin{array}{c} m-1\\ j-1\end{array}\right]_q,\] it follows that
 all these are polynomials in $q$ with integer coefficients.

\medskip Let $M=(E,\rho)$ be a combinatroid. Set $n=|E|$ and $k=\rho(E)$.
Following~\cite{JurrPellik-z}, for $1\leq r\leq n$, we define the
$r$-generalized Hamming weight enumerator
\[W^{(r)}(x,y,q):=\frac{1}{\langle r\rangle_q}\sum_{j=0}^r
    \left[\begin{array}{c}r\\ j\end{array}\right]_q(-1)^{r - j}\,q^{r-j\choose 2}
    (x - y)^{n - k}\, y^k\, T_M(\frac{x}{y},\frac{x +(q^j - 1)y}{x-y}).\]
\begin{conjecture}\label{rW-Jurrius} Let $M=(E,\rho)$ be a combinatroid.
Set $n=|E|$ and $k=\rho(E)$. Then \begin{equation}\label{conj-rw-Jurrius} T_M(x,
y)=
    x^n(x - 1)^{k - n}\sum_{r=0}^{n-k}\left(\prod_{j=0}^{r-1}((x - 1)(y - 1) - q^j)
    \right)W^{(r)}(1,1/x,q).\end{equation}
\end{conjecture}

\begin{remark} When $M$ is the associated matroid to a linear code, via its
parity check matrix, this conjecture has been proved by Jurrius~\cite[Thm.
3.3.5]{MThJurrius-z}.\end{remark}

\begin{example} Let $C$ be the binary linear $[6,3]$ code with parity check
matrix
\[
H=\left[
\begin{array}{rrrrrr}
1&1&1&1&0&0\\
1&1&1&0&1&0\\
1&1&1&0&0&1
\end{array}
\right]\!\!.\] (See Example $C_1$ of Section 5.2 in \cite{MThJurrius-z}) The
bases of the matroid $M[H]$ are \[\{145, 146, 156, 245, 246, 256, 345, 346, 356,
456\},\] and its Tutte polynomial is
\[T(x,y)=x + x^2 + x^3 + y + x y + x^2 y + y^2 + x y^2 + x^2 y^2 +
    y^3.\]
\begin{eqnarray*}
 W^{(0)}(x,y,t)&=&x^6;\\
W^{(1)}(x,y,t)&=&3 x^4y^2 + (-2 + t) x^3 y^3 + 3 x^2 y^4 +
        3 (-2 + t) x y^5 + (3 - 3 t + t^2) y^6;\\
W^{(2)}(x,y,t)&=&x^3y^3 + 3 x y^5 + (-3 + t + t^2)
y^6;\\
W^{(3)}(x,y,t)&=&y^6.
\end{eqnarray*}
Substituting these $W^{(r)}$'s in Eq. (\ref{conj-rw-Jurrius}) we recover
$T(x,y)$.
\end{example}

\begin{example} Let $C$ be the binary linear $[6,3]$ code with parity check
matrix
\[
H=\left[
\begin{array}{rrrrrr}
1&0&0&1&0&0\\
0&1&0&0&1&0\\
0&0&1&0&0&1
\end{array}
\right]\!\!.\] (See Example $C_1$ of Section 5.2 in \cite{MThJurrius-z}) The
bases of the matroid $M[H]$ are \[\{123, 126, 135, 156, 234, 246, 345, 456\},\]
and its Tutte polynomial is
\[T(x,y)=x^3 + 3 x^2 y + 3 x y^2 + y^3.\]
 \begin{eqnarray*}
W^{(0)}(x,y,t)&=&x^6;\\
W^{(1)}(x,y,t)&=&3 x^4 y^2 +
    3 (-1 + t) x^2 y^4 + (-1 + t)^2 y^6;\\
W^{(2)}(x,y,t)&=&3 x^2 y^4 + (-2 + t + t^2) y^6;\\
W^{(3)}(x,y,t)&=&y^6.
 \end{eqnarray*} Substituting these $W^{(r)}$'s in Eq.
(\ref{conj-rw-Jurrius}) we recover $T(x,y)$.
\end{example}

\begin{example} Let $C$ be the binary Hamming linear $[7,4]$ code with parity
check matrix
\[
H=\left[
\begin{array}{rrrrrrr}
0&1&1&1&1&0&0\\
1&0&1&1&0&1&0\\
1&1&0&1&0&0&1
\end{array}
\right]\!\!.\] The Tutte polynomial of $M[H]$ is
\[T(x,y)=3 x + 4 x^2 + x^3 + 3 y + 7 x y + 6 y^2 + 3 y^3 + y^4.\]
 \begin{eqnarray*}
 W^{(0)}(x,y,t)&=&x^7;\\
 W^{(1)}(x,y,t)&=&7 x^4y^3 + 7 x^3 y^4 + 21 (-2 + t) x^2 y^5 +
        7(6 - 5 t + t^2) x y^6\\
        &&\ + (-13 + 15 t - 6 t^2 +
              t^3) y^7;\\
 W^{(2)}(x,y,t)&=&21 x^2y^5 +
        7 (-5 + t + t^2) x y^6 + (15 - 6 t - 5 t^2 + t^3 +
              t^4) y^7;\\
W^{(3)}(x,y,t)&=&7 x y^6+ (-6 + t + t^2 + t^3) y^7;\\
W^{(4)}(x,y,t)&=&y^7.
 \end{eqnarray*}
Substituting these $W^{(r)}$'s in Eq. (\ref{conj-rw-Jurrius}) we recover
$T(x,y)$.
\end{example}

\begin{example}
Let $M$ be the demimatroid in Example~\ref{Contra-full}.
\[\begin{array}{|r|c|c|c|c|c|c|c|c|}\hline
 X&\ \emptyset&\ 1&\ 2&\ 3&12&13&23&E\\\hline\hline
 \rho&0&0&0&0&1&1&1&2\\\hline
 \rho^*&0&0&0&0&0&0&0&1\\\hline
 \rho^\circ&0&1&1&1&1&1&1&1\\\hline
 \rho^\cd&0&1&1&1&2&2&2&2\\\hline
 \end{array}.\]
\[T_M(x,y)=x - 2 x^2 + y - 3 x y + 3 x^2 y.\]
\[W_M(x,y,t)=x^3 + 3 (t-1) x^2 y + 3 (1 - t) x y^2 + (t-1) y^3.\]
\begin{eqnarray*}
W^{(0)}(x,y,t)&=&(x-y)^3x^3;\\
W^{(1)}(x,y,t)&=&(x - y)^3 y (3 x^2 - 3 x y + y^2);\\
W^{(2)}(x,y,t)&=&0.
\end{eqnarray*}
\[x^6(x-y)^{-4}[W^{(0)}(1,1/x,t)+((x-1)(y-1)-1)W^{(1)}(1,1/x,t)]=T_M(x,y).\]
\end{example}


\end{document}